\UseRawInputEncoding
\documentclass{{amsart}}

\usepackage[english]{babel}

\usepackage[letterpaper,top=2cm,bottom=2cm,left=3cm,right=3cm,marginparwidth=1.75cm]{geometry}

\usepackage{amsmath,amsthm,amssymb}
\usepackage{graphicx}
\graphicspath{ {./images/} }
\usepackage{color}
\usepackage{manyfoot}
\usepackage{enumerate}

\newtheorem{theorem}{Theorem}[section]
\newtheorem{corollary}[theorem]{Corollary}
\newtheorem{lemma}[theorem]{Lemma}

\newtheorem{conjecture}[theorem]{Conjecture}

\newtheorem{remark}[theorem]{Remark}

\newcommand{\be}{\begin{equation}}
\newcommand{\ee}{\end{equation}}
\newcommand{\lt}{\left}
\newcommand{\rt}{\right}

\newcommand{\goto}{\rightarrow}

\newcommand{\la}{\lambda}
\newcommand{\og}{\omega}
\newcommand{\w}{\wedge}
\newcommand{\e}{\varepsilon}

\numberwithin{equation}{section}

\title{Closed minimal hypersurfaces in $\mathbb S^5$ with constant $S$ and $A_3$  }
\author{Joel Spruck}
\address{
Department of Mathematics, Johns Hopkins University, Baltimore, Maryland 21218 US.
}
\email{jspruck1@jhu.edu}
\author{Ling Xiao}
\address{
Department of Mathematics,  University of  Connecticut,  Storrs, Connecticut 06269 US.
}
\email{ling.2.xiao@uconn.edu}

\begin{document}
\maketitle

\begin{abstract}
In this paper, we prove that a closed minimally immersed hypersurface $M^4\subset\mathbb S^5$ with constant $S:=\sum\limits_{i=1}^4\lambda_i^2$ and $A_3:=\sum\limits_{i=1}^4\lambda_i^3$ whose scalar curvature
$R_M$ is nonnegative must be isoparametric. Moreover, $S$ can only be $0, 4,$ and $12.$ That is $M^4$ is either an equatorial $4$-sphere, a clifford torus, or a Cartan's minimal hypersurface.
\end{abstract}

\section{Introduction}
Let $M^n$ be a closed minimally immersed hypersurface of $\mathbb S^{n+1},$ and $h$ be its second fundamental form. We denote the square of the length of $h$ by $S.$ It is well-known that the extrinsically defined quantity $S$ is in fact intrinsic and satisfies $S=n(n-1)-R_M,$ where $R_M$ is the scalar curvature of $M^n.$

In the papers \cite{Che68} and \cite{CdK70}, S. S. Chern proposed the following famous conjecture.
\begin{conjecture}[Chern's Conjecture 1] Let $M^n\subset\mathbb S^{n+1}$ be a closed minimally immersed hypersurface in the unit sphere with constant scalar curvature $R_M$. Then for each $n$, the set of all possible values of $R_M$ is discrete.
\end{conjecture}

The only known examples of closed minimal hypersurfaces $M^n\subset\mathbb S^{n+1}$ with constant scalar curvature are all isoparametric, that is have constant principal curvatures. The study of isoparametric hypersurfaces in space forms was initiated by Cartan in a remarkable series of papers \cite{Car38, Car39.1, Car39.2, Car40}. After a long hiatus, M\"unzner's fundamental papers \cite{Mun80, Mun81} became the basis of  the modern study of isoparametric hypersurfaces in spheres and their  classification is now complete (see \cite{Chi18}) but is extremely complicated. The focus of  research on the Chern conjectures is basically concentrated on the revised conjecture

\begin{conjecture}[Chern's Conjecture 2] Let $M^n\subset\mathbb S^{n+1}$ be a closed minimally immersed hypersurface in the unit sphere with constant scalar curvature $R_M$. Then $M^n$ is isoparametric.
\end{conjecture}

There have been many works on the Chern conjectures. In particular, Simons \cite{Sim68} proved that if $0\leq S\leq n,$ then either $S\equiv 0$ or $S\equiv n.$ By the result of Chern-do Carmo-Kobayashi \cite{CdK70} and Lawson \cite{Law69} we know that the only closed minimal hypersurfaces in $\mathbb S^{n+1}$ with $S=n$ are the Clifford tori. When $n=3,$ that is for a closed minimal hypersurface $M^3$ immersed in $\mathbb S^4$ with constant scalar curvature, Peng-Terng \cite{PT83} showed that if $S\geq 3$ then either $S=3$ or $S\geq 6.$ Later  de Almeida-Brito \cite{AB90} proved that this same class of hypersurfaces are isoparametric if $S\leq 6$. In 1993, Chang \cite{Chang93} settled the complementary case and completed the classification theorem for closed minimal hypersurfaces $M^3\subset\mathbb S^4$ with constant scalar curvature.

For the higher dimensional case, much less is known. In 2005, Lusala-Scherfner-Sousa \cite{LSS05} generalized the theorem of \cite{AB90} and proved that a closed minimal Willmore hypersurface \vspace{-.1in} (i.e., $A_3:=\sum\limits_i\la_i^3\equiv 0$) $M^4\subset \mathbb S^5$ with nonnegative constant scalar curvature must be isoparametric, that is the principal curvatures are constant. Building on this result, Deng-Gu-Wei \cite{DGW17} obtained the classification theorem for minimal Willmore hypersurfaces in $\mathbb S^5$ with constant scalar curvature. The theorem of \cite{AB90} was further generalized by Tang-Wei-Yan \cite{TWY20} who proved that if a closed minimal hypersurface $M^n\subset \mathbb S^{n+1}$ has constant $A_k:=\sum\limits_i\la_i^k, 2\leq k\leq n-1$ curvatures, nonnegative scalar curvature, and distinct principal curvatures everywhere, then $M^n$ is isoparametric.

In order to generalize the theorem of \cite{AB90}, both papers \cite{LSS05} and \cite{TWY20} need some additional assumptions on the curvatures of $M^n.$ More specifically, in \cite{LSS05} the authors assumed $A_3=0$ instead of a general constant, while in \cite{TWY20} the authors assumed the principal curvatures are distinct on the entire $M^n$. The key reasons for these assumptions is that when adapting the argument of \cite{AB90}, one needs a smooth test function on $M^n$ (see \cite[Lemma 1]{AB90}) and the derivative of the test function needs to have the right sign (see \cite[page 203, inequalities below equation (6.2)]{AB90}). For $n\geq 4$ it is difficult to find a test function satisfying both these  conditions.

In this paper, we introduce a new way to construct  this test function.  Moreover, we modify the argument in \cite{AB90} and no longer need to use \cite[Lemma 1]{AB90}. These innovations enable us to generalize the result in \cite{LSS05} to the case $A_3=\text{constant}$ without assuming the principal curvatures are distinct. We prove the following theorems.
\begin{theorem}\label{thm1}
A closed minimally immersed hypersurface $M^4\subset\mathbb S^5$ with constant $S:=\sum\limits_{i=1}^4\lambda_i^2$ and $A_3:=\sum\limits_{i=1}^4\lambda_i^3$ whose scalar curvature $R_M$ is nonnegative must be isoparametric. Moreover, $S$ can only be $0, 4,$ and $12.$ That is $M^4$ is either an equatorial $4$-sphere, a clifford torus, or  Cartan's minimal hypersurface.
\end{theorem}

When there exists a point $p\in M^4$ such that at $p$ there are exactly two distinct principal curvatures, we obtain a stronger result without any sign restriction on $R_M$.
\begin{theorem}\label{thm2}
Let $M^4\subset\mathbb S^5$ be an immersed closed minimal hypersurface with constant $S$ and $A_3.$ If there exists a point $p\in M^4$ such that there are exactly two distinct principal curvatures at $p$, then $S=4,$ $A_3=0,$ and $M^4$ is a Clifford torus.
\end{theorem}

Theorem \ref{thm2} is proved in Section \ref{p2} and Theorem \ref{thm1} is proved in Section \ref{p1}. We also note that in this paper we will always arrange $\la_i$ such that $\la_4\geq\la_3\geq\la_2\geq\la_1.$ For convenience of exposition we will  assume that $A_3\geq 0$  ;  otherwise
replace $\la_i $ by $\tilde{\la}_i=-\la_i$.

\section{Proof of Theorem \ref{thm2}}
\label{p2}
In order to prove Theorem \ref{thm2} we will need the following two Lemmas.

\begin{lemma}\cite[Proposition 2.1]{Ok}
\label{lem-p2-1}
Let $a_1, \cdots, a_n$ be real numbers satisfying $\sum_{i=1}^n a_i=0$. Then
\[\lt|\sum_{i=1}^n a_i^3\rt|\leq\frac{n-2}{\sqrt{n(n-1)}}(\sum_{i=1}^n a_i^2)^{3/2}\]
with equality  if and only if   $n-1$ of  the $a_i$'s  are equal.
\end{lemma}

\begin{lemma}\cite[Lemma 2.2]{Ok}
\label{lem-p2-2}
Suppose that $\varphi$ is a smooth function such that $\la_1\geq\varphi$ everywhere and
$\la_1=\varphi$ at $p$. Let $\mu$ denote the multiplicity of the smallest curvature eigenvalue at $p$, so
that $\la_1=\cdots=\la_\mu<\la_{\mu+1}\leq\cdots\leq\la_n.$ Then at $p$, $\nabla_ih_{kl}=\nabla_i\varphi\delta_{kl}$ for $1\leq k, l\leq\mu.$
\end{lemma}

\begin{proof}[\textbf{Proof of Theorem \ref{thm2}}]
Case 1: There exists $p\in M^4$ such that at $p$ we have $\la_4=\la_3=\la>0$ and $\la_1=\la_2=-\la<0.$ Then we have $A_3=0$ on $M^4.$
The conclusion follows from \cite[Theorem 3.1]{DGW17}.\\
Case 2: There exists $p\in M^4$ such that at $p$ we have $\la_4=3\la>0$ and $\la_3=\la_2=\la_1=-\la<0.$
Then at $p$ we have $A_3=\sum\la_i^3=\frac1{\sqrt{3}} S^\frac{3}{2}.$ Since $A_3$ and $S$ are constant,  Lemma \ref{lem-p2-1} with $n=4$
implies $\la_4=3\la$ and $\la_3=\la_2=\la_1=-\la$ everywhere on $M^4$ for some fixed $\la>0,$ that is $M^4$ is isoparametric. In this case we know that the principal curvatures $\la_i,$ $1\leq i\leq 4$ are smooth on $M^4$ which implies $h_{jji}=0$ for any $1\leq i, j\leq 4.$ Moreover, applying Lemma \ref{lem-p2-2} we find
$h_{ijk}=0$ for any $1\leq i, j, k\leq 4.$ A straightforward calculation (see \cite{PT83} for example) for closed minimal hypersurfaces with constant square norm of second fundamental form $S$ gives
\[\Delta S=2(4-S)S+2\sum\limits_{i,j,k}h^2_{ijk}=0.\]
This implies $S=0$ or $4.$ That is, $M^4$ is either an equatorial $4$-sphere or a Clifford torus, which leads to a contradiction. Therefore, Theorem \ref{thm2} is proved.
\end{proof}

\bigskip
\section{Proof of Theorem \ref{thm1}}
\label{p1}
In this section, we will prove Theorem \ref{thm1}. In view of Theorem \ref{thm2} and Lemma \ref{lem-p2-1} we only need to prove the following result.
\begin{theorem}
\label{thm3}
Let $M^4\subset\mathbb S^5$ be an immersed closed minimal hypersurface with constant $S=\sum\limits_i\la_i^2\in(4, 12]$ and $A_3=\sum\limits_i\la_i^3\in [0, \frac1{\sqrt{3}}S^{3/2}).$ Then $M^4$ is isoparametric with
$S=12$ and $A_3=0.$ That is $M^4$ is a Cartan's minimal hypersurface.
\end{theorem}

\subsection{Preliminaries}
\label{pre} In order to construct a nice auxiliary function to work with, we need the following Lemma.
\begin{lemma}\cite[Lemma 3.1]{Guan02}
\label{lem-pre-1}
For all $\delta>0$ there is an even function $h(t)\in C^{\infty}(\mathbb R)$ such that

(i) $h(t)\geq |t|$ for all $t\in\mathbb R,$ $h(t)=|t|$ for all $|t|\geq\delta;$

(ii) $|h'(t)|\leq 1$ and $h''(t)\geq 0$ for all $t\in\mathbb R$ and $h'(t)\geq 0$ for all $t\geq 0.$
\end{lemma}

 By virtue of Theorem \ref{thm2} we also know that under the assumption of Theorem \ref{thm3} there is no point on $M^4$ at which there are exactly two distinct principal curvatures. Therefore, for any $p\in M^4,$ there are either three or four distinct principal curvatures at $p$. Moreover by our assumption that $A_3\geq0,$ it is easy to see that if there are three distinct principal curvatures then we either have $\la_3=\la_2$ or $\la_2=\la_1.$ For if $\la_4=\la_3$, then
$\la_2=-(\la_1+2\la_4)$ and so $\la_1^2+2\la_1 \la_4+3\la_4^2-\frac12S=0$. Then
\begin{eqnarray*}
&\la_1=-\la_4-\sqrt{\frac12S-2{\la_4}^2}\\
&\la_2=-\la_4+\sqrt{\frac12S-2{\la_4}^2}\\
&\la_3=\la_4
\end{eqnarray*}
which gives
\[\sum_{i=1}^4 {\la_i}^3=-6\la_4(\frac12 S-2\la_4^2)<0,\]
a contradiction.\\

 By reduction of variables, it is easy to see that there are  a finite number of
4-tuples $\{\la_1, \la_2, \la_3, \la_4\}$  satisfying
\[
\tag{I}\left\{
\begin{aligned}
&\sum\limits_{i=1}^4\la_i=0,\,\,\sum\limits_{i=1}^4\la_i^2=S,\,\,\sum\limits_{i=1}^4\la_i^3=A_3,\\
&\la_3-\la_2=\la_2-\la_1,
\end{aligned}\right.\]
where $S, A_3$ are constants.
Similarly there are  a finite number of 4-tuples $\{\la_1, \la_2, \la_3, \la_4\}$  satisfying
\[
\tag{II}\left\{
\begin{aligned}
&\sum\limits_{i=1}^4\la_i=0,\,\,\sum\limits_{i=1}^4\la_i^2=S,\,\,\sum\limits_{i=1}^4\la_i^3=A_3,\\
&\la_3=\la_2,
\end{aligned}\right.\]
and
\[
\tag{III}\left\{
\begin{aligned}
&\sum\limits_{i=1}^4\la_i=0,\,\,\sum\limits_{i=1}^4\la_i^2=S,\,\,\sum\limits_{i=1}^4\la_i^3=A_3,\\
&\la_2=\la_1,
\end{aligned}\right.\]
respectively. Now, define $f=(\la_3-\la_2)^2$ and $g=(\la_2-\la_1)^2.$  By our assumption that at every point of $M^4$ there are either three  or four distinct principle curvatures,  both $f$ and $g$ cannot both vanish for any point of $M^4$. Hence, there exists $\e_0=\e_0(S, A_3)>0$ so that
 \be \label{prop1}
f+g\geq \, 2\e_0 .
\ee
everywhere on $M^4$.
%
 Moreover, $f-g \not \equiv 0$ on all of $M^4$ for otherwise $M^4$ is isoparametric and must be the Cartan's isoparametric minimal hypersurface in $\mathbb S^5$ which is Willmore, i.e $A_3=0.$ This is contrary to our assumption that $f=g$.
Thus $f-g=0$ only on at most a finite number of disjoint closed subsets $\Gamma_i$ of  $M^4$ satisfying (I) and we can choose small disjoint tubes
 $\Gamma_i(r)$
  and define
\[ \e_1(r)=\sup_{\cup \Gamma_i(r)}|f-g| ,\,\, \e_2(r)=\inf_{M^4-\cup  \Gamma_i(r)}|f-g| ~.\]
Since both $\e_1(r)$ and $\e_2(r)$ tend to $0$ as $r\goto 0$, we can choose $r$ so small that $\max{(\e_2(r),\e_1(r))}< \e_0$.
Now let $\delta=\min{(\e_2(r),\e_1(r))}$ and define
\[K(p):=\frac{f(p)+g(p)}{2}-\frac{h(f(p)-g(p))}{2}\]
where $h$ is the function given in Lemma \ref{lem-pre-1}. Clearly $|f-g| \geq \delta $ on $M^4-\cup \Gamma_i(r)$ and so
\be \label{prop2} K=f+g-|f-g|=\min{(f,g)}\geq 0 \,\,\mbox{on $M^4-\cup \Gamma_i(r)$}~.\ee
Equation \eqref{prop2} also holds at points of $\cup \Gamma_i(r)$ where $|f-g|\geq \delta$ and so $K=\min{(f,g)}\geq \e_0$ by \eqref{prop1}.
At points of $\cup \Gamma_i(r)$  where $|f-g|<\delta,\, h(f-g)\leq \delta <\e_0$
and so by \eqref{prop1} and Lemma \ref{lem-pre-1}.
\be \label{prop3} K \geq 2\e_0- \delta \geq 2\e_0-\e_0=\e_0\,\,\mbox{on $\cup \Gamma_i(r)$} ~.\ee

\begin{remark} If $\{f=g\}=\emptyset$ then $|f-g|\geq \delta_0$ for some small $\delta_0=\delta_0(S,A_3)>0$. Then choosing
$\delta=\frac{\delta_0}2$ in Lemma \ref{lem-pre-1} gives that $K=\min{(f,g)}$ everywhere on $M^4$.
\end{remark}

Thus we have shown
\begin{lemma} \label{prop4} $K \geq 0$ on $M^4$ and $K=\min{(f,g)}$ outside a small neighborhood of
$\{p\in M^4: f(p)=g(p)\}$.  Moreover, since $f+g \geq 2\e_0$ there exists $\delta_1=\delta_1(S, A_3)<\e_0$ so that  $f\geq \e_0$ on the  set
 $\{0<g<\delta_1\}$ and $g\geq \e_0$ on the set $\{0<f<\delta_1\}$.
\end{lemma}

In the rest of this paper we set
\[Y=\{p\in M^4: K(p)>0\},\]
\[Y_1^c=\{p\in M^4: g(p)=0\},\]
and
\[Y_2^c=\{p\in M^4: f(p)=0\}.\]
It is clear that $f,\,g,\,$ and $K$ are smooth on $Y$ and $Y^c=Y_1^c\cup Y_2^c.$

\begin{corollary}
\label{cor-pre-2}
The set $Y\neq\emptyset$ and $\text{vol}(Y)>0.$
\end{corollary}
\begin{proof}
It follows immediately from Lemma \ref{prop4} that $Y\neq\emptyset.$ Since $K$ is a continuous function on $M^4,$ we know $Y$ is an open set  and
so $\text{vol}(Y)>0$.
\end{proof}

\subsection{Structure equations}
\label{str} In this subsection, we will look at the structure equations on the open subset $Y$ of $M^4.$ By our assumption, the principal curvatures $\{\la_1, \la_2, \la_3, \la_4\}$ are simple on $Y.$ Let $\{e_1, e_2, e_3, e_4\}$ be the corresponding eigenvectors; then locally $\{e_1, e_2, e_3, e_4\}$ form an oriented orthonormal frame fields on $Y.$
Let $\{\og_1, \og_2, \og_3, \og_4\}$ be the dual frame. Then one has the following structure equations on $M^4$:
\be\label{str1}
d\og_i=\sum\limits_j\og_{ij}\w\og_j, \,\, \og_{ij}+\og_{ji}=0,
\ee
and
\be\label{str2}
d\og_{ij}=\sum\limits_k\og_{ik}\w\og_{kj}-\frac{1}{2}\sum\limits_{k, l}R_{ijkl}\og_k\w\og_l.
\ee
Here $R_{ijkl}$ is the curvature tensor of the induced metric on $M^4$. Let $II=\sum\limits_{i, j}h_{ij}\og_i\otimes\og_j$ denote the second fundamental form of $M^4$, then
\[R_{ijkl}=\delta_{ik}\delta_{jl}-\delta_{il}\delta_{jk}+h_{ik}h_{jl}-h_{il}h_{jk}.\]
Moreover by the definition of the covariant derivative of the second fundamental form, we have
\be\label{str3}
\og_{ij}=\sum\limits_k\frac{h_{ijk}\og_k}{\lambda_i-\lambda_j}, \,\,i\neq j.
\ee

\subsection{The $3$-form $\Phi$}
\label{tf}
Differentiating $\sum\la_j=0,$ $\sum\la_j^2=S$ and $\sum\la_j^3=A_3$ with $S, A_3$ being constants gives
\[\sum\limits_jh_{jji}=0,\,\,\sum\limits_j\la_jh_{jji}=0,\,\,\mbox{and $\sum\limits_j\la_j^2h_{jji}=0.$}\]
Therefore, for any fixed $1\leq i\leq 4$ we obtain using Cramer's rule and Vandermonde determinants
\be\label{tf1}
\begin{aligned}
h_{11i}&=-\frac{(\la_4-\la_3)(\la_4-\la_2)}{(\la_3-\la_1)(\la_2-\la_1)}h_{44i},\\
h_{22i}&=\frac{(\la_4-\la_3)(\la_4-\la_1)}{(\la_2-\la_1)(\la_3-\la_2)}h_{44i},\\
h_{33i}&=-\frac{(\la_4-\la_2)(\la_4-\la_1)}{(\la_3-\la_1)(\la_3-\la_2)}h_{44i}.
\end{aligned}
\ee
Following the notations in \cite[page 7]{CL24}, we define
\be\label{tf2}
\begin{aligned}
\theta_{12}&=\og_3\w\og_4\w\og_{12},\,\,\theta_{13}&=\og_4\w\og_2\w\og_{13},\,\,\theta_{14}&=\og_2\w\og_3\w\og_{14},\\
\theta_{23}&=\og_1\w\og_4\w\og_{23},\,\,\theta_{24}&=\og_3\w\og_1\w\og_{24},\,\,\theta_{34}&=\og_1\w\og_2\w\og_{34}.
\end{aligned}
\ee
We introduce a $3$-form $\Phi$ on $Y$ as follows:
\be\label{tf3}
\Phi=\sum\limits_{i<j}\theta_{ij}.
\ee
In the following two lemmas we will compute $d\Phi$ on $Y.$ The calculations are long and tedious but straightforward.
Let  $d\theta_{ij}=X_{ij}\text{vol},$
where \text{vol} is the volume form of $M^4$.

\begin{lemma}
\label{lem-tf-1}
Let $d\theta_{12}=X_{12}\text{vol}$. Then on the set $Y$ we have
\be\label{tf-12}
\begin{aligned}
X_{12}&=\frac{(\la_3-\la_4)[(\la_1-\la_3)^2(\la_2-\la_3)-(\la_1-\la_4)^2(\la_2-\la_4)]h^2_{441}}{(\la_1-\la_2)^2(\la_1-\la_3)^2(\la_2-\la_3)^2}\\
&+\frac{(\la_3-\la_4)[(\la_2-\la_3)^2(\la_1-\la_3)-(\la_2-\la_4)^2(\la_1-\la_4)]h^2_{442}}{(\la_1-\la_2)^2(\la_1-\la_3)^2(\la_2-\la_3)^2}\\
&+\frac{(\la_1-\la_4)(\la_2-\la_4)(\la_3-\la_4)^2h^2_{443}}{(\la_1-\la_2)^2(\la_1-\la_3)^2(\la_2-\la_3)^2}\\
&+\frac{(\la_1-\la_3)(\la_2-\la_3)(\la_3-\la_4)^2h^2_{444}}{(\la_1-\la_2)^2(\la_1-\la_3)^2(\la_2-\la_3)^2}\\
&+\frac{2h^2_{123}}{(\la_1-\la_3)(\la_2-\la_3)}+\frac{2h^2_{124}}{(\la_1-\la_4)(\la_2-\la_4)}-R_{1212}.
\end{aligned}
\ee
\end{lemma}
\begin{proof}
In view of subsection \ref{str} we find
\be\label{tf4}
\begin{aligned}
d\theta_{12}&=(d\og_3)\w\og_4\w\og_{12}-\og_3\w(d\og_4)\w\og_{12}+\og_3\w\og_4\w(d\og_{12})\\
&=\og_{3k}\w\og_k\w\og_4\w\og_{12}-\og_3\w\og_{4k}\w\og_k\w\og_{12}+\og_3\w\og_4\w(\og_{1k}\w\og_{k2}-R_{1212}\,\og_1\w\og_2)\\
&=\og_{31}\w\og_1\w\og_4\w\og_{12}+\og_{32}\w\og_2\w\og_4\w\og_{12}-\og_{3}\w\og_{41}\w\og_1\w\og_{12}\\
&-\og_{3}\w\og_{42}\w\og_2\w\og_{12}+\og_{3}\w\og_4\w\og_{13}\w\og_{32}+\og_{3}\w\og_4\w\og_{14}\w\og_{42}-R_{1212}\text{vol}\\
&=\textcircled{1}+\textcircled{2}-\textcircled{3}-\textcircled{4}+\textcircled{5}+\textcircled{6}-R_{1212}\text{vol}.
\end{aligned}
\ee
Applying equation \eqref{str3} we derive
\be\label{tf5}
\begin{aligned}
\textcircled{1}&=\lt[\frac{h_{312}\og_2+h_{313}\og_3}{\la_3-\la_1}\rt]\w\og_1\w\og_4\w\lt[\frac{h_{122}\og_2+h_{123}\og_3}{\la_1-\la_2}\rt]\\
&=\lt[\frac{h^2_{123}}{(\la_3-\la_1)(\la_1-\la_2)}-\frac{h_{331}h_{221}}{(\la_3-\la_1)(\la_1-\la_2)}\rt]\text{vol}.
\end{aligned}
\ee

\be\label{tf6}
\begin{aligned}
\textcircled{2}&=\lt[\frac{h_{321}\og_1+h_{323}\og_3}{\la_3-\la_2}\rt]\w\og_2\w\og_4\w\lt[\frac{h_{123}\og_3+h_{121}\og_1}{\la_1-\la_2}\rt]\\
&=\lt[\frac{-h^2_{123}}{(\la_3-\la_2)(\la_1-\la_2)}+\frac{h_{332}h_{112}}{(\la_3-\la_2)(\la_1-\la_2)}\rt]\text{vol}.
\end{aligned}
\ee

\be\label{tf7}
\begin{aligned}
\textcircled{3}&=\og_3\w\lt(\frac{h_{412}\og_2+h_{414}\og_4}{\la_4-\la_1}\rt)\w\og_1\w\lt(\frac{h_{124}\og_4+h_{122}\og_2}{\la_1-\la_2}\rt)\\
&=\lt[\frac{-h^2_{124}}{(\la_4-\la_1)(\la_1-\la_2)}+\frac{h_{441}h_{221}}{(\la_4-\la_1)(\la_1-\la_2)}\rt]\text{vol}.
\end{aligned}
\ee

\be\label{tf8}
\begin{aligned}
\textcircled{4}&=\og_3\w\lt(\frac{h_{421}\og_1+h_{424}\og_4}{\la_4-\la_2}\rt)\w\og_2\w\lt(\frac{h_{124}\og_4+h_{121}\og_1}{\la_1-\la_2}\rt)\\
&=\lt[\frac{h^2_{124}}{(\la_4-\la_2)(\la_1-\la_2)}-\frac{h_{442}h_{112}}{(\la_4-\la_2)(\la_1-\la_2)}\rt]\text{vol}.
\end{aligned}
\ee

\be\label{tf9}
\begin{aligned}
\textcircled{5}&=\og_3\w\og_4\w\lt(\frac{h_{132}\og_2+h_{131}\og_1}{\la_1-\la_3}\rt)\w\lt(\frac{h_{321}\og_1+h_{322}\og_2}{\la_3-\la_2}\rt)\\
&=\lt[\frac{-h^2_{123}}{(\la_1-\la_3)(\la_3-\la_2)}+\frac{h_{113}h_{223}}{(\la_1-\la_3)(\la_3-\la_2)}\rt]\text{vol}.
\end{aligned}
\ee

\be\label{tf10}
\begin{aligned}
\textcircled{6}&=\og_3\w\og_4\w\lt(\frac{h_{141}\og_1+h_{142}\og_2}{\la_1-\la_4}\rt)\w\lt(\frac{h_{421}\og_1+h_{422}\og_2}{\la_4-\la_2}\rt)\\
&=\lt[\frac{-h^2_{124}}{(\la_1-\la_4)(\la_4-\la_2)}+\frac{h_{114}h_{224}}{(\la_1-\la_4)(\la_4-\la_2)}\rt]\text{vol}.
\end{aligned}
\ee
Lemma \ref{lem-tf-1} follows from combining \eqref{tf1} with \eqref{tf4}--\eqref{tf10}.
\end{proof}

Similarly, we can show (see also \cite[page 7 and 8]{CL24})
\begin{lemma}
\label{lem-tf-2}
Let $d\theta_{ij}=X_{ij}\text{vol},$ where $(i, j)\neq(1, 2)$ and $1\leq i, j\leq 4.$ Then on the set $Y$ we have
\be\label{tf-13}
\begin{aligned}
X_{13}&=-\frac{(\la_2-\la_4)[(\la_1-\la_2)^2(\la_2-\la_3)+(\la_1-\la_4)^2(\la_3-\la_4)]h^2_{441}}{(\la_1-\la_2)^2(\la_1-\la_3)^2(\la_2-\la_3)^2}\\
&+\frac{(\la_2-\la_4)[(\la_1-\la_2)(\la_2-\la_3)^2-(\la_1-\la_4)(\la_3-\la_4)^2]h^2_{443}}{(\la_1-\la_2)^2(\la_1-\la_3)^2(\la_2-\la_3)^2}\\
&+\frac{(\la_1-\la_4)(\la_2-\la_4)^2(\la_3-\la_4)h^2_{442}}{(\la_1-\la_2)^2(\la_1-\la_3)^2(\la_2-\la_3)^2}\\
&-\frac{(\la_1-\la_2)(\la_2-\la_4)^2(\la_2-\la_3)h^2_{444}}{(\la_1-\la_2)^2(\la_1-\la_3)^2(\la_2-\la_3)^2}\\
&-\frac{2h^2_{123}}{(\la_1-\la_2)(\la_2-\la_3)}+\frac{2h^2_{134}}{(\la_1-\la_4)(\la_3-\la_4)}-R_{1313},
\end{aligned}
\ee

\be\label{tf-14}
\begin{aligned}
X_{14}&=\frac{(\la_2-\la_3)[(\la_1-\la_3)^2(\la_3-\la_4)-(\la_1-\la_2)^2(\la_2-\la_4)]h^2_{441}}{(\la_1-\la_2)^2(\la_1-\la_3)^2(\la_2-\la_3)^2}\\
&-\frac{(\la_3-\la_4)h^2_{442}}{(\la_1-\la_2)^2(\la_1-\la_3)}-\frac{(\la_2-\la_4)h^2_{443}}{(\la_1-\la_2)(\la_1-\la_3)^2}\\
&+\frac{(\la_2-\la_3)[(\la_1-\la_2)(\la_2-\la_4)^2-(\la_1-\la_3)(\la_3-\la_4)^2]h^2_{444}}{(\la_1-\la_2)^2(\la_1-\la_3)^2(\la_2-\la_3)^2}\\
&-\frac{2h^2_{124}}{(\la_1-\la_2)(\la_2-\la_4)}-\frac{2h^2_{134}}{(\la_1-\la_3)(\la_3-\la_4)}-R_{1414},
\end{aligned}
\ee

\be\label{tf-23}
\begin{aligned}
X_{23}&=\frac{(\la_1-\la_4)^2(\la_2-\la_4)(\la_3-\la_4)h^2_{441}}{(\la_1-\la_2)^2(\la_1-\la_3)^2(\la_2-\la_3)^2}\\
&+\frac{(\la_1-\la_4)^2h^2_{444}}{(\la_1-\la_2)(\la_1-\la_3)(\la_2-\la_3)^2}\\
&-\frac{(\la_1-\la_4)[(\la_2-\la_4)^2(\la_3-\la_4)+(\la_1-\la_2)^2(\la_1-\la_3)]h^2_{442}}{(\la_1-\la_2)^2(\la_1-\la_3)^2(\la_2-\la_3)^2}\\
&-\frac{(\la_1-\la_4)[(\la_2-\la_4)(\la_3-\la_4)^2+(\la_1-\la_2)(\la_1-\la_3)^2]h^2_{443}}{(\la_1-\la_2)^2(\la_1-\la_3)^2(\la_2-\la_3)^2}\\
&+\frac{2h^2_{123}}{(\la_1-\la_2)(\la_1-\la_3)}+\frac{2h^2_{234}}{(\la_2-\la_4)(\la_3-\la_4)}-R_{2323},
\end{aligned}
\ee

\be\label{tf-24}
\begin{aligned}
X_{24}&=-\frac{(\la_3-\la_4)h^2_{441}}{(\la_1-\la_2)^2(\la_2-\la_3)}+\frac{(\la_1-\la_4)h^2_{443}}{(\la_1-\la_2)(\la_2-\la_3)^2}\\
&+\frac{(\la_1-\la_3)[(\la_3-\la_4)(\la_2-\la_3)^2-(\la_1-\la_4)(\la_1-\la_2)^2]h^2_{442}}{(\la_1-\la_2)^2(\la_1-\la_3)^2(\la_2-\la_3)^2}\\
&-\frac{(\la_1-\la_3)[(\la_2-\la_3)(\la_3-\la_4)^2+(\la_1-\la_2)(\la_1-\la_4)^2]h^2_{444}}{(\la_1-\la_2)^2(\la_1-\la_3)^2(\la_2-\la_3)^2}\\
&+\frac{2h^2_{124}}{(\la_1-\la_2)(\la_1-\la_4)}-\frac{2h^2_{234}}{(\la_2-\la_3)(\la_3-\la_4)}-R_{2424},
\end{aligned}
\ee

\be\label{tf-34}
\begin{aligned}
X_{34}&=\frac{(\la_2-\la_4)h^2_{441}}{(\la_1-\la_3)^2(\la_2-\la_3)}+\frac{(\la_1-\la_4)h^2_{442}}{(\la_1-\la_3)(\la_2-\la_3)^2}\\
&+\frac{(\la_1-\la_2)[(\la_2-\la_3)^2(\la_2-\la_4)-(\la_1-\la_3)^2(\la_1-\la_4)]h^2_{443}}{(\la_1-\la_2)^2(\la_1-\la_3)^2(\la_2-\la_3)^2}\\
&+\frac{(\la_1-\la_2)[(\la_2-\la_3)(\la_2-\la_4)^2-(\la_1-\la_3)(\la_1-\la_4)^2]h^2_{444}}{(\la_1-\la_2)^2(\la_1-\la_3)^2(\la_2-\la_3)^2}\\
&+\frac{2h^2_{134}}{(\la_1-\la_3)(\la_1-\la_4)}+\frac{2h^2_{234}}{(\la_2-\la_3)(\la_2-\la_4)}-R_{3434}.
\end{aligned}
\ee
\end{lemma}

Combining Lemma \ref{lem-tf-1} and Lemma \ref{lem-tf-2} we obtain
\be\label{form-derivative}
d\Phi=\lt(\sum\limits_{i=1}^4L_ih^2_{44i}-\frac{1}{2}R_M\rt)\text{vol}\,\,\mbox{on $Y$.}
\ee
If we denote $\gamma:=(\la_2-\la_1)^2(\la_3-\la_1)^2(\la_3-\la_2)^2$ then
\[\gamma L_1=(\la_4-\la_3)[(\la_3-\la_1)^2(\la_3-\la_2)-(\la_4-\la_2)(\la_4-\la_1)^2]-(\la_4-\la_2)(\la_3-\la_2)(\la_2-\la_1)^2,\]
\[\gamma L_2=(\la_4-\la_3)[(\la_3-\la_2)^2(\la_3-\la_1)-(\la_4-\la_1)(\la_4-\la_2)^2]-(\la_4-\la_1)(\la_3-\la_1)(\la_2-\la_1)^2,\]
\[\gamma L_3=(\la_2-\la_1)[(\la_3-\la_2)^2(\la_4-\la_2)-(\la_4-\la_1)(\la_3-\la_1)^2]-(\la_4-\la_1)(\la_4-\la_2)(\la_4-\la_3)^2,\]
and
\[\gamma L_4=(\la_2-\la_1)[(\la_4-\la_2)^2(\la_3-\la_2)-(\la_3-\la_1)(\la_4-\la_1)^2]-(\la_3-\la_1)(\la_3-\la_2)(\la_4-\la_3)^2.\]
Recall that on $Y$ we have $\la_4>\la_3>\la_2>\la_1.$ Therefore, it is clear that $L_i<0$ for $1\leq i\leq 4.$

\subsection{Computing $dK\w\Phi$}
\label{dk}
\label{com}
We will compute $dK\w\Phi$ on the set $\{0<K<\delta_1\}\cap Y,$ where $\delta_1>0$ is the constant chosen in Lemma \ref{prop4}.
\begin{lemma}
\label{dk-lem1}
Let $g=(\la_2-\la_1)^2,$ and $\delta_1>0$ be the constant chosen in Lemma \ref{prop4}. Then on the set $\{p\in M^4: 0<g(p)<\delta_1\},$ we have
\be\label{dk1}
dg\w\Phi\leq C\lt(\sum\limits_{i=1}^4h_{44i}^2\rt)\text{vol}
\ee
for some $C=C(S, A_3, \delta_1)>0.$
\end{lemma}
\begin{proof}
A direct calculation yields the following equalities.
\be\label{dk2}
\begin{aligned}
\og_1\w\Phi&=\og_1\w(\theta_{12}+\theta_{13}+\theta_{14})\\
&=\lt[-\frac{(\la_4-\la_3)(\la_4-\la_1)}{(\la_3-\la_2)(\la_2-\la_1)^2}+\frac{(\la_4-\la_2)(\la_4-\la_1)}{(\la_3-\la_2)(\la_3-\la_1)^2}
-\frac{1}{\la_4-\la_1}\rt]h_{441}\text{vol},
\end{aligned}
\ee

\be\label{dk3}
\begin{aligned}
\og_2\w\Phi&=\og_2\w(\theta_{12}+\theta_{23}+\theta_{24})\\
&=\lt[\frac{(\la_4-\la_2)(\la_4-\la_1)}{(\la_3-\la_1)(\la_3-\la_2)^2}-\frac{1}{\la_4-\la_2}-\frac{(\la_4-\la_3)(\la_4-\la_2)}{(\la_3-\la_1)(\la_2-\la_1)^2}\rt]
h_{442}\text{vol},
\end{aligned}
\ee

\be\label{dk4}
\begin{aligned}
\og_3\w\Phi&=\og_3\w(\theta_{13}+\theta_{23}+\theta_{34})\\
&=\lt[-\frac{1}{\la_4-\la_3}+\frac{(\la_4-\la_3)(\la_4-\la_1)}{(\la_3-\la_2)^2(\la_2-\la_1)}-\frac{(\la_4-\la_3)(\la_4-\la_2)}{(\la_3-\la_1)^2(\la_2-\la_1)}\rt]
h_{443}\text{vol},
\end{aligned}
\ee

\be\label{dk5}
\begin{aligned}
\og_4\w\Phi&=\og_4\w(\theta_{14}+\theta_{24}+\theta_{34})\\
&=\lt[-\frac{(\la_4-\la_3)(\la_4-\la_2)}{(\la_3-\la_1)(\la_2-\la_1)(\la_4-\la_1)}+\frac{(\la_4-\la_3)(\la_4-\la_1)}{(\la_2-\la_1)(\la_3-\la_2)(\la_4-\la_2)}\right.\\
&\left.-\frac{(\la_4-\la_2)(\la_4-\la_1)}{(\la_3-\la_1)(\la_3-\la_2)(\la_4-\la_3)}\rt]h_{444}\text{vol}.
\end{aligned}
\ee

Now, in view of equation \eqref{tf1} we have
\be\label{dk6}
\begin{aligned}
g_i&=2(\la_2-\la_1)(h_{22i}-h_{11i})\\
&=2(\la_4-\la_3)\lt(\frac{\la_4-\la_1}{\la_3-\la_2}+\frac{\la_4-\la_2}{\la_3-\la_1}\rt)h_{44i}
:=m_0h_{44i},
\end{aligned}
\ee
where $m_0>0$ is a uniformly bounded function on the set $\{p\in M^4: 0<g(p)<\delta_1\}.$
Therefore, it is easy to see that
\be\label{dk7}
\begin{aligned}
dg\w\Phi&=\sum\limits_{i=1}^4g_i\og_i\w\Phi\\
&=\lt(B_{1}^gh_{441}^2+B_{2}^gh^2_{442}+\sum\limits_{i=1}^4G_{i}^gh^2_{44i}\rt)\text{vol}.
\end{aligned}
\ee
Here, $B_{1}^g=-m_0\frac{(\la_4-\la_3)(\la_4-\la_1)}{(\la_3-\la_2)(\la_2-\la_1)^2}<0,$
$B^{g}_2=-m_0\frac{(\la_4-\la_3)(\la_4-\la_2)}{(\la_3-\la_1)(\la_2-\la_1)^2}<0,$ and $G^{g}_i,$ $1\leq i\leq 4,$ are uniformly bounded functions on the set $\{0<g<\delta_1\}.$
This completes the proof of this Lemma.
\end{proof}

Similarly, we consider the function $f=(\la_3-\la_2)^2.$ A direct calculation yields
\be\label{dk8}
\begin{aligned}
f_i&=2(\la_3-\la_2)(h_{33i}-h_{22i})\\
&=2(\la_4-\la_1)\lt(-\frac{\la_4-\la_2}{\la_3-\la_1}-\frac{\la_4-\la_3}{\la_2-\la_1}\rt)h_{44i}
:=m_1h_{44i},
\end{aligned}
\ee
where $m_1<0$ is a uniformly bounded function on the set $\{p\in M^4: 0<f(p)<\delta_1\}.$
In conjunction with equations \eqref{dk2}--\eqref{dk5} we conclude the following lemma.
\begin{lemma}
\label{dk-lem2}
Let $f=(\la_3-\la_2)^2,$ and $\delta_1>0$ be the constant chosen in Lemma \ref{prop4}. Then on the set $\{p\in M^4: 0<f(p)<\delta_1\},$ we have
\be\label{dk9}
df\w\Phi\leq C\lt(\sum\limits_{i=1}^4h_{44i}^2\rt)\text{vol}
\ee
for some $C=C(S, A_3, \delta_1)>0.$
\end{lemma}

\subsection{Conclusion}
\label{con}
Recall that that $Y\neq\emptyset$ by Corollary \ref{cor-pre-2}.  For any $\e\in (0, \delta_1/2)$ we set
\[
\begin{aligned}
Y&=X_\e\cup Y_\e\cup Z_\e\\
&=\{p\in M^4: 0<g(p)<\e\}\cup\{p\in M^4: K(p)\geq\e\}\cup\{p\in M^4: 0<f(p)<\e\}.
\end{aligned}
\]

Following \cite{AB90}, for any smooth function $\eta: (0, \infty)\rightarrow \mathbb R$ with $\eta\equiv 0$ in a small neighborhood of $0$, we may apply Stokes's Theorem to
\[d((\eta\circ K)\Phi)=(\eta\circ K)d\Phi+(\eta'\circ K)dK\w\Phi\]
and obtain
\be\label{con1}
\int_Y(\eta\circ K)d\Phi+\int_Y(\eta'\circ K)dK\w\Phi=0.
\ee
Given a small $\e>0,$ we choose a smooth function $\eta=\eta_\e: \mathbb R\rightarrow \mathbb R$ such that\\
(a) $0\leq\eta_\e\leq 1,$\\
(b) $\eta_\e=0$ for $t\leq\frac{\e}{3},$\\
(c) $\eta_\e=1$ for $t\geq\e,$\\
(d) $0\leq\eta'_\e\leq O\lt(\frac{1}{\e}\rt)$ for $t\in (\frac{\e}{3}, \e)$
and $\eta'_\e\equiv 0$ at everywhere else.\\

In view of equations \eqref{form-derivative}, \eqref{con1}, and the assumption that $R_M\geq 0,$  we obtain
\be\label{con2}
\begin{aligned}
0&\geq\int_Y(\eta_\e\circ K)\lt(\sum\limits_{i=1}^4L_ih^2_{44i}-\frac{1}{2}R_M\rt)\text{vol}=-\int_Y(\eta_\e'\circ K)dK\w\Phi\\
&=-\int_{X_\e}(\eta'_\e\circ g)dg\w\Phi-\int_{Z_\e}(\eta'_\e\circ f)df\w\Phi\\
&\geq-\frac{C}{\e}\int_{X_\e\cup Z_\e}\sum\limits_{i=1}^4h^2_{44i}\text{vol}.
\end{aligned}
\ee
By Lemma \ref{dk-lem1}, Lemma \ref{dk-lem2} and the choice of $\eta_{\e}$, we know that $C>0$ is a constant independent of $\e.$
It is  well known that the $4$-dimensional minimal hypersurface in $\mathbb S^5$ with constant scalar curvature satisfies
\[\sum\limits_{i, j, k}h^2_{ijk}=S(S-4).\]
This yields that $|h_{ijk}|$ is bounded for any $1\leq i, j, k\leq 4.$ In view of \eqref{tf1} we have

\[|h_{11i}|=\lt|\frac{(\la_4-\la_3)(\la_4-\la_2)}{(\la_3-\la_1)(\la_2-\la_1)}h_{44i}\rt|\,\,\text{ is bounded on $X_\e,$}\]
and
\[ |h_{33i}|= \lt|\frac{(\la_4-\la_2)(\la_4-\la_1)}{(\la_3-\la_1)(\la_3-\la_2)}h_{44i}\rt|\,\,\text{ is bounded on $Z_\e.$}\]

This in turn implies that $|h_{44i}|\leq O(\sqrt{\e})$ on $X_\e\cup Z_\e.$ Equation \eqref{con2} then becomes
\be\label{con3}
\begin{aligned}
0&\geq\int_Y(\eta_\e\circ K)\lt(\sum\limits_{i=1}^4L_ih^2_{44i}-\frac{1}{2}R_M\rt)\text{vol}\\
&\geq-C\lt[\text{vol}(X_\e)+\text{vol}(Z_\e)\rt]
\end{aligned}
\ee
for some constant $C>0$ that is independent of $\e.$
It is easy to see that for any $\e_1>\e_2>0$ we have $X_{\e_2}\subset X_{\e_1}$ and $Z_{\e_2}\subset Z_{\e_1}.$
Moreover, we have $$\bigcap X_\e=\bigcap Z_\e=\emptyset.$$
Therefore, as $\e\rightarrow 0,$ we find $\text{vol}(X_\e)\rightarrow 0$ and $\text{vol}(Z_\e)\rightarrow 0.$ Together with \eqref{con3} this implies
\be\label{con4}\int_Y(\eta_\e\circ K)\lt(\sum\limits_{i=1}^4L_ih^2_{44i}-\frac{1}{2}R_M\rt)\text{vol}\rightarrow 0\,\,\mbox{as $\e\rightarrow 0.$}\ee

Recall that $\text{vol}(Y)>0$ by Corollary \ref{cor-pre-2} . Combining this with the fact that $L_i<0$ and $R_M\geq 0$ on $Y$ we conclude that on the set $Y,$
$R_M\equiv 0$ and $h_{44i}\equiv 0$ for $i\leq 4.$ Applying equation \eqref{tf1} again we obtain $h_{jji}=0$ on $Y$ for any $1\leq i, j\leq 4.$
This yields $\la_1, \la_2, \la_3, \la_4$ are constants on $Y.$ By continuity of $\la_i, 1\leq i\leq 4,$ we obtain $Y^c=\emptyset.$
Therefore, we conclude that $M^4$ is isoparametric and $R_M\equiv 0,$ that is $S\equiv 12.$ That is $M^4$ is a Cartan's minimal hypersurface.

\end{document}